\documentclass[11pt,twoside,a4paper]{article}
\usepackage[T1]{fontenc}
\usepackage[utf8]{inputenc}

\usepackage[australian]{babel}

\usepackage[driver=pdftex,margin=3cm,heightrounded=true,centering,includeheadfoot]{geometry}

\usepackage{amssymb,amsmath,amsfonts}
\usepackage{amsthm}
\usepackage[all]{xy}

\allowdisplaybreaks
\numberwithin{equation}{section}

\usepackage[dvipsnames]{xcolor}
\definecolor{MyBlue}{cmyk}{1,0.13,0,0.63}
\definecolor{MyGreen}{cmyk}{0.91,0,0.88,0.52}
\definecolor{MyRed}{rgb}{.6,0,0}
\newcommand{\mylinkcolor}{MyBlue}
\newcommand{\mycitecolor}{MyGreen}
\newcommand{\myurlcolor}{MyRed}

\makeatletter
\def\@endtheorem{\endtrivlist}% NEW
\makeatother
\theoremstyle{plain}
\newtheorem{thm}{Theorem}[section]

\newtheorem*{main*}{Main Theorem}
\newtheorem{lemma}[thm]{Lemma}
\newtheorem{prop}[thm]{Proposition}
\newtheorem{coro}[thm]{Corollary}
\theoremstyle{definition}
\newtheorem{df}[thm]{Definition}
\newtheorem{remark}[thm]{Remark}
\newtheorem{example}[thm]{Example}
\newtheorem{assumption}[thm]{Assumption}

\usepackage{hyperref}
\usepackage{cleveref}

\hypersetup{%
  bookmarksnumbered=true,bookmarksopen=false,%
  plainpages=false,% necessary to prevent duplicate page identifiers
  linktocpage=true,%
  colorlinks=true,breaklinks=true,%
  linkcolor=\mylinkcolor,citecolor=\mycitecolor,urlcolor=\myurlcolor,%
  pdfpagelayout=OneColumn,%
  pageanchor=true,%
}

\makeatletter
\def\thm@space@setup{%
  \thm@preskip=4pt plus 2pt minus 2pt
  \thm@postskip=\thm@preskip
}
\renewenvironment{proof}[1][\proofname]{\par
  \pushQED{\qed}%
  \normalfont \topsep4\p@\relax % NEW
  \trivlist
  \item[\hskip\labelsep
        \itshape
    #1\@addpunct{.}]\ignorespaces
}{%
  \popQED\endtrivlist\@endpefalse
}
\makeatother

\usepackage[shortlabels]{enumitem}
\setlist{topsep=4pt plus 2pt minus 2pt,partopsep=0pt,itemsep=2pt plus 2pt minus 2pt,parsep=0.5\parskip}
\setenumerate[1]{label={\upshape(\arabic*)}}

\usepackage{fancyhdr}
\newcommand{\MR}[1]{}

\let\OLDthebibliography\thebibliography
\renewcommand\thebibliography[1]{
  \addcontentsline{toc}{section}{\refname}
  \OLDthebibliography{#1}
  \setlength{\parskip}{0pt}
  \setlength{\itemsep}{0pt plus 0.3ex}
}

\def\<{\left\langle}
\def\>{\right\rangle}
\def\Rp{{\mathbb{R}_{\geq0}}}
\def\C{{\mathbb{C}}}
\def\R{{\mathbb{R}}}
\def\N{{\mathbb{N}}}
\def\mat#1#2#3#4{{\begin{pmatrix} #1&#2\\ #3&#4 \end{pmatrix}}}

\def\famop{\{A(t)\}_{t\in [0,T]}}

\def\LH{L^2([0,T],\mH)} 

\def\ddt{\frac{d}{dt}}

\def\pdt{\frac{\partial}{\partial t}}
\def\pds{\frac{\partial}{\partial s}}

\newcommand{\mH}{\mathcal{H}}
\newcommand{\mW}{\mathcal{W}}
\newcommand{\D}{\mathcal{D}}
\newcommand{\Id}{\mathrm{Id}}
\newcommand{\into}{\hookrightarrow}
\renewcommand{\hat}{\widehat}
\renewcommand{\tilde}{\widetilde}

\DeclareMathOperator{\Dom}{Dom}
\DeclareMathOperator{\Ran}{Ran}
\DeclareMathOperator{\Ker}{Ker}
\DeclareMathOperator{\Coker}{Coker}
\DeclareMathOperator{\APS}{APS}
\DeclareMathOperator{\ind}{ind}
\DeclareMathOperator{\sfl}{sf}
\DeclareMathOperator{\spann}{span}
\DeclareMathOperator{\Dim}{Dim}
\DeclareMathOperator{\ev}{ev}
\DeclareMathOperator{\Real}{Re}
\DeclareMathOperator{\spec}{spec}
\DeclareMathOperator{\supp}{supp}

\usepackage{slashed}
\newcommand{\sD}{\slashed{D}}

\def\MyTitle{The APS-index and the spectral flow}
\title{\MyTitle}
\author{
Koen van den Dungen and Lennart Ronge\footnote{The authors would like to thank the Hausdorff Center for Mathematics and the Bonn International Graduate School of Mathematics (BIGS-M) for their support.}
\\[2mm]
{\small Mathematisches Institut}, 
{\small Universit\"at Bonn}\\
{\small Endenicher Allee 60, D-53115 Bonn}\\
}

\begin{document}

\maketitle

\begin{abstract}
\noindent
We study the Atiyah-Patodi-Singer (APS) index, and its equality to the spectral flow, in an abstract, functional analytic setting. More precisely, we consider a (suitably continuous or differentiable) family of self-adjoint Fredholm operators $A(t)$ on a Hilbert space, parametrised by $t$ in a finite interval. We then consider two different operators, namely $D := \frac{d}{dt}+A$ (the abstract analogue of a Riemannian Dirac operator) and $D := \frac{d}{dt}-iA$ (the abstract analogue of a Lorentzian Dirac operator). The latter case is inspired by a recent index theorem by Bär and Strohmaier (Amer.\ J.\ Math. 141 (2019), 1421--1455) for a Lorentzian Dirac operator equipped with APS boundary conditions. In both cases, we prove that Fredholm index of the operator $D$ equipped with APS boundary conditions  is equal to the spectral flow of the family $A(t)$. 

\vspace{\baselineskip}

\end{abstract}

\tableofcontents

\section{Introduction}

In a recent paper \cite{BS19}, Bär and Strohmaier derived a Lorentzian version of the Atiyah-Patodi-Singer (APS) index theorem for globally hyperbolic spacetimes with future and past spacelike boundaries. 
The main purpose of this article is to recast their work in a more abstract, functional analytic setting. For the sake of completeness and comparison, we will also discuss the ordinary Riemannian version of the Atiyah-Patodi-Singer index theorem in this abstract setting. 

Consider an even-dimensional, oriented, time-oriented Lorentzian spin manifold $(X,g)$. We will assume that $(X,g)$ is \emph{globally hyperbolic}, which implies \cite[Theorem 1.1]{BS05} that it is isometric to $(\R\times\Sigma,-N^2dt^2+g_t)$, where the \emph{Cauchy hypersurface} $\Sigma$ is a smooth manifold with a family of Riemannian metrics $\{g_t\}_{t\in\R}$, and the \emph{lapse function} $N$ is a smooth function $\R\times\Sigma\to(0,\infty)$. 
Furthermore, as in \cite{BS19} we assume that the Cauchy hypersurface $\Sigma$ is compact (for the noncompact case, see \cite{Bra20}). 

Let $\nu$ be the past-directed unit normal vector field, and let $\beta = \gamma(\nu)$ be Clifford multiplication by $\nu$. 
Since $X$ is even-dimensional, the spinor bundle decomposes into spinors of positive and negative chirality. Identifying the positive and negative chirality spinors using $\beta$, the Lorentzian Dirac operator takes the form \cite [Eq. (11)]{vdD18_FST}
\[
\sD = \mat{0}{-\nabla_\nu + i A(t) - \tfrac n2 H}{-\nabla_\nu - i A(t) - \tfrac n2 H}{0} , 
\]
where $A = \{A(t)\}_{t\in\R}$ is the family of Dirac operators on the Cauchy hypersurfaces $\{t\}\times\Sigma$, and $H$ is the mean curvature of the hypersurfaces $\{t\}\times\Sigma$. 
If $X$ is a metric product (i.e.\ $g = -dt^2 + g_0$, where the metric $g_0$ on $\Sigma$ is independent of $t$), then we have $N\equiv1$, $H\equiv0$, and $A(t)=A_0$, and we obtain 
\[
\sD = \mat{0}{\partial_t + i A_0}{\partial_t - i A_0}{0} . 
\]
In general, if $X$ is not a metric product, we can use parallel transport along the curves $\R\to X$ given by $t\mapsto(t,x)$ for some $x\in \Sigma$, to show that $\sD$ is unitarily equivalent to \cite[Proposition III.5]{vdD18_FST}
\begin{align*}
\sD \simeq \mat{0}{N^{-\frac12} \partial_t N^{-\frac12} + i B(t)}{N^{-\frac12} \partial_t N^{-\frac12} - i B(t)}{0} , 
\end{align*}
where $B(t)$ is obtained from $A(t)$ via the parallel transport isomorphism. 
The bottom left corner of the Dirac operator shall be denoted
\[
D := -\nabla_\nu - i A(t) - \tfrac n2 H \simeq N^{-\frac12} \partial_t N^{-\frac12} - i B(t) . 
\]

We now restrict the globally hyperbolic spacetime $X=\R\times\Sigma$ to a finite time interval $[0,T]$. Thus we consider the globally hyperbolic spacetime $M := [0,T]\times\Sigma$, with past and future spacelike boundaries $\{0\}\times\Sigma$ and $\{T\}\times\Sigma$ (respectively). 
Since these spacelike boundaries are Riemannian manifolds, we can use the Dirac operators $A(0)$ and $A(T)$ to define Atiyah-Patodi-Singer (APS) boundary conditions (i.e., the domain is restricted to those functions $f$ with $f(0)$ in the range of the negative spectral projection of $A(0)$ and $f(T)$ in the range of the positive spectral projection of $A(T)$). Thus equipping $D$ with APS boundary conditions, Bär and Strohmaier \cite{BS19} then prove that the resulting operator $D_{\APS}$ is Fredholm, and that its index can be computed by the same formula as in the original (Riemannian) Atiyah-Patodi-Singer index theorem \cite{APS75}. A crucial step in their proof is to relate this index to the spectral flow $\sfl(A)$ of the family $\{A(t)\}_{t\in[0,T]}$ of Dirac operators on the Cauchy hypersurfaces: 
\begin{align}
\label{eq:ind-sf}
\ind(D_{\APS})=\sfl(A) . 
\end{align}
Suppose now that our globally hyperbolic spacetime $M := [0,T]\times\Sigma$ is of product form near the boundary. Then in particular the lapse function $N$ is equal to $1$ near the boundary, so multiplication by $N^{\frac12}$ preserves the APS boundary conditions. Hence we can consider the new operator 
\[
N^{\frac12} D_{\APS} N^{\frac12} \simeq \partial_t - i N^{\frac12} B(t) N^{\frac12} . 
\]
Thus, writing $\tilde A(t) := N^{\frac12} B(t) N^{\frac12}$, we can summarise the above as follows: we wish to study the Fredholm index of an operator of the form $\partial_t - i \tilde A(t)$ with APS boundary conditions. 
The purpose of this article is to rederive Eq.\ \eqref{eq:ind-sf} for such operators in a more general functional analytic setting:
\begin{itemize}
\item $A=\famop$ is a strongly continuously differentiable family of self-adjoint Fredholm operators on a Hilbert space $\mH$ with constant domain $W$; 
\item $D_{\APS}$ is the closure of the operator 
\[
D := \ddt-iA
\]
on $L^2([0,T],\mH)$, equipped with APS boundary conditions.
\end{itemize}
Furthermore, for the sake of completeness and comparison, we will also discuss the `Riemannian' analogue, namely the operator $\ddt+A$ with APS boundary conditions. 

Let us briefly summarise the contents of this article. 
First, in Section \ref{sec:families}, some basic facts regarding strongly continuously differentiable families of operators will be derived for later use. In Section \ref{sec:sf}, we review the notion of spectral flow, following \cite{Phi96}. 
We prove in Theorem \ref{flowind} that the spectral flow of a norm-continuous family $A=\famop$ is equal to the relative index of the pair $(P_{<0}(0),P_{<0}(T))$ of negative spectral projections of $A$ at the endpoints, provided that $(P_{<0}(0),P_{<0}(t))$ is a Fredholm pair for each $t\in[0,T]$. This generalises a known result \cite[Theorem 3.6]{Les05} in the special case where $P_{<0}(0)-P_{<0}(t)$ is compact (cf.\ Remark \ref{remark:flowind}). 

In Section \ref{Riem}, we describe the abstract analogue of the \emph{Riemannian} APS-index. We note that, on a Riemannian manifold $M=[0,T]\times\Sigma$ with the product metric $g = dt^2 + g_0$, the Dirac operator is of the form 
\[
\sD = \mat{0}{-\partial_t + A_0}{\partial_t + A_0}{0} ,
\]
where $A_0$ denotes the Dirac operator on the hypersurface $\Sigma$. We consider in Section \ref{Riem} the more general setting where $A=\famop$ is a norm-continuous family of self-adjoint operators on a Hilbert space $\mH$ with constant domain $W$, where the inclusion $W\into\mH$ is compact. We then study the operator
\[
D := \ddt + A 
\]
equipped with APS boundary conditions. We can extend $A$ to a family $\tilde A$ on the whole real line. We then recall from \cite{APS76} the classical `index = spectral flow' result:
\begin{align}
\label{eq:ind=sf}
\ind\left(\ddt+\tilde A\right) = \sfl(\tilde A) .
\end{align}
This equality has been rigorously proven by Robbin and Salamon \cite{RS95} for a suitable differentiable family of operators $\tilde A = \{\tilde A(t)\}_{t\in\R}$. 
In fact, the assumption of differentiability is not necessary, and Eq.\ \eqref{eq:ind=sf} remains valid for norm-continuous families (see \cite[Theorem 2.1]{AW11} and \cite[Theorem 5.2]{vdD19_Index_DS}). 
We will prove (see Theorem \ref{thm:Riem_APS_index_sfl}) that the operator $D_{\APS}$ is Fredholm, and that we also have the equality
\[
\ind(D_{\APS}) = \sfl(A) . 
\]
The proof is based on relating the index of $D_{\APS}$ (on the interval $[0,T]$ with APS boundary conditions) to the index of the extension $\ddt+\tilde A$ (on the complete line $\R$). The main issue to overcome is that Eq.\ \eqref{eq:ind=sf} is only valid for families with invertible endpoints, and we show that we may always perturb $A$ to a family with invertible endpoints, without changing its spectral flow or the index of $D_{\APS}$. 

In Section \ref{sec:APS-ind}, we finally describe the abstract analogue of the \emph{Lorentzian} APS-index. In this case, we consider the operator 
\[
D := \ddt - i A 
\]
equipped with APS boundary conditions. 
Here we need to assume in addition that $A$ is \emph{strongly continuously differentiable}.
The additional $-i$ before the family $A$ leads to qualitatively very different behavior of the operator $D$. 
For instance, for the operator $\ddt + A$ on the real line, both the norm of $\ddt f$ and of $Af$ can be estimated by the graph norm $\|f\|_{\ddt+A}$ (cf.\ \cite{RS95}). For the operator $D=\ddt-iA$, however, the equation $Df=0$ has solutions with arbitrarily large $\ddt f$ and $Af$. In fact, the equation has a unique solution for any initial value: indeed, the Cauchy problem corresponding to $D$ is well-posed (see Theorem \ref{thm:Cauchy}). Moreover, solutions to $Df=0$ will not be square-integrable on $\R$, which necessitates restricting to a finite interval $[0,T]$ (and introducing boundary conditions). 

In Section \ref{sec:evol}, we introduce the evolution operator $Q$, which describes solutions to the initial value problem
\begin{align*}
Df&=0, \qquad
f(s)=x.
\end{align*}
The construction of this evolution operator, following \cite[Ch.\ 5]{Pazy83}, requires the assumption that $A$ is \emph{strongly continuously differentiable}. 
We then use the evolution operator in Section \ref{sec:APS-ind_proj} to relate the index of $D_{\APS}$ to the (relative) index of a certain Fredholm pair of spectral projections at the endpoints of the interval, corresponding to the family of `evolved' operators
$$\hat A(t)=Q(0,t)A(t)Q(t,0).$$
We show in section \ref{sec:APS-ind_sf} that $\hat A$ is again strongly continuously differentiable and therefore norm-continuous. In particular, we then know from Theorem \ref{flowind} that the (relative) index of the pair of spectral projections of $\hat A(0)$ and $\hat A(T)$ is equal to the spectral flow of $A$. Thus we combine our results to prove the main theorem:
\begin{main*}
If $(D|_{[0,t]})_{\APS}$ is Fredholm for all $t\in[0,T]$, we have
$$\ind(D_{\APS})=\sfl(A).$$
\end{main*}
Here $D|_{[0,t]}$ is the `restriction' of $D$ to the interval $[0,t]$.
The Lorentzian Dirac operator studied in \cite{BS19} satisfies the hypothesis of our main theorem. 
The general idea and some parts of the proof of our main theorem are similar as in \cite{BS19}, while other parts are different. In particular, the use of Fredholm pairs and the aforementioned spectral projections allows for a much wider generalisation than a straightforward adaptation of the arguments of \cite{BS19} would. 

Finally, Section \ref{sec:counterexample} will describe a counterexample which shows that Fredholmness of $(D|_{[0,t]})_{\APS}$ is not a consequence of the other assumptions.

This article is largely based on the Master's thesis by the second author (\cite{Ron19}), advised by Matthias Lesch and the first author. Several proofs which are only sketched in this article, can be found in more detail in \cite{Ron19}.

The authors would like to thank Matthias Lesch for interesting discussions and for his helpful comments on this manuscript. 

\subsection*{Notation}
Let $\mH$ denote a separable, infinite-dimensional Hilbert space. 
For an operator $T$ on $\mH$ and subspaces $X,Y\subset\mH$ satisfying $X\subset\Dom T$ and $\Ran T\subset Y$, 
we denote by $T|_{X\to Y}$ the restriction of $T$ to $X$ with codomain $Y$. 

Integrals and $L^p$-spaces of Banach-space-valued functions should be understood in the sense of Bochner integration (for details, see e.g.\ \cite[Ch.\ 3]{Hille-Phillips96}).

\section{Families of Operators}
\label{sec:families}

For this whole section, let $X$, $Y$ and $Z$ be Banach spaces, and let $J$ be a compact interval. 
A family of operators $S\colon J\rightarrow B(X,Y)$ is called strongly continuous, if it is continuous with respect to the strong operator topology on $B(X,Y)$. It is called strongly continuously differentiable, if it is differentiable with respect to the strong operator topology and the derivative is strongly continuous.
Explicitly, this means that there exists a strongly continuous family $S'\colon J\to B(X,Y)$ such that for each $x\in X$ we have $\frac{d}{dt} \big( S(t) x \big) = S'(t) x$. 

By the Banach-Steinhaus Theorem (or Uniform Boundedness Principle), strongly continuous families are uniformly bounded. As composition is continuous with respect to the strong topology when restricted to bounded subsets (in the operator norm), the composition of two strongly continuous families is again strongly continuous.

\begin{lemma}
\label{lem:diff}
Let $S\colon J\to B(X,Y)$ be strongly continuously differentiable. Then the following statements hold:
\begin{enumerate}
\item 
\label{diffbound}
$S$ is norm-continuous. 
\item 
\label{idiff}
If $S(t)$ is invertible for all $t\in J$, then the family
\begin{align*}
S^{-1}\colon J&\rightarrow B(Y,X) , \qquad
t\mapsto S(t)^{-1}
\end{align*}
is strongly continuously differentiable with derivative $-S^{-1}S'S^{-1}$.
\end{enumerate}
\end{lemma}
\begin{proof}
Norm-continuity at $t\in J$ is a consequence of the Banach-Steinhaus Theorem applied to 
$$F:=\left\{\left.\frac{1}{s-t}(S(s)-S(t)) \,\right|\, s\in J \backslash\{t\}\right\}.$$
As the inversion map is norm-continuous as well, also $S^{-1}$ is norm-continuous (and in particular uniformly bounded).
Let $t\in J$ and $h\in\R$ small enough such that $t+h \in J$. Then for $y\in Y$ we have 
\begin{align*}
(S(t+h)^{-1}-S(t)^{-1})y&=S(t+h)^{-1}(S(t)-S(t+h))S(t)^{-1}y\\
&=-S(t+h)^{-1}\big(hS'(t)S(t)^{-1}y+o(h)\big)\\
&=-hS(t)^{-1}S'(t)S(t)^{-1}y +o(h),
\end{align*}
which proves the second statement. 
\end{proof}

\begin{prop}
\label{prop:diff}
Let $X_0\subseteq X$ be a dense subspace, and let $Y_0\subseteq Y$ be a subspace of $Y$ with a stronger norm that turns it into a Banach space. 
Consider a family of operators $S\colon J\to B(X,Y_0)$. 
\begin{enumerate}
\item 
\label{prodiff}
Suppose $S\colon J\to B(X,Y_0)$ is strongly continuous and $S\colon J\to B(X,Y)$ is strongly continuously differentiable. If $R\colon J\to B(Y,Z)$ is a strongly continuous family that restricts to a strongly differentiable family in $B(Y_0,Z)$, then $t\mapsto R(t)S(t)$ is strongly continuously differentiable in $B(X,Z)$, with derivative $R'(t)S(t)+R(t)S'(t)$.
\item 
\label{strictdiff}
Suppose that the restriction $S|_{X_0}\colon J\to B(X_0,Y)$ is strongly continuously differentiable, such that the derivative extends to a strongly continuous family $S'\colon J\to B(X,Y_0)$. 
Then $S\colon J\to B(X,Y_0)$ is also strongly continuously differentiable with derivative $S'$.
\end{enumerate}
\end{prop}
\begin{proof}
Since Banach-Steinhaus guarantees uniform boundedness of the differential quotients, composition can be treated like a continuous bilinear map and the calculation used to show the usual product rule can be applied verbatim to prove the first statement. 

Regarding the second statement, for $x\in X_0$, we have
$$S(t)x=S(t_0)x+\int\limits_{t_0}^tS'(r)xdr,$$
with the integral being taken in $Y$. Since the integrals in $Y_0$ and $Y$ coincide and both sides are bounded linear functions of $x$, we get the same equality in $Y_0$ for any $x\in X$.
\end{proof} 

\begin{remark}
The above proposition asserts in particular that pointwise compositions of strongly continuously differentiable families of operators are again strongly continuously differentiable.
We also note that this implies the analogous result for evaluation instead of composition: if $f\colon [0,T]\rightarrow Y$ is a function, we can set $S\colon [0,T]\rightarrow B(\C,Y)$, $S(t)(1)=f(t)$. Then $R(t)f(t)$ has the same regularity properties as  $R(t)S(t)$ in the strong topology. 
\end{remark}

\section{Spectral flow}
\label{sec:sf}

\begin{assumption}
\label{set_sf}
Let $\mH$ be a separable Hilbert space, let $W\subseteq \mH$ be a dense subspace, and let $\famop$ be a family of unbounded self-adjoint Fredholm operators on $\mH$ with constant domain $W$. 
We equip $W$ with the graph norm of $A(0)$. 
We \emph{assume} that the family $A$ is a norm-continuous map from $[0,T]$ to $B(W,H)$.\footnote{We note here that $W$ is complete (since $A(0)$ is closed), and it is then a consequence of the closed graph theorem that automatically $A(t)\in B(W,\mH)$ for each $t\in[0,T]$.}
\end{assumption}

The notion of spectral flow for a path of self-adjoint operators was first defined by Atiyah and Lusztig, and it appeared in the work of Atiyah, Patodi, and Singer \cite[\S7]{APS76}. 
Heuristically, the spectral flow of the family $A$ counts the number of eigenvalues of $A(t)$ (counted with multiplicities) crossing $0$ as $t$ varies from $0$ to $T$, i.e.\ the number of negative eigenvalues becoming positive minus that of positive eigenvalues becoming negative. 
In this article we will follow the analytic definition of spectral flow given by Phillips in \cite{Phi96}. 

\begin{df}
Consider an interval $I\subset \R$, and let $\chi_I$ denote the characteristic function of $I$.
For $t\in[0,T]$, consider the spectral projection of $A(t)$ and the corresponding spectral subspace given by 
\begin{align*}
P_I(t) &:= \chi_{I}(A(t)) , & 
\mH_I(t) &:= \Ran(P_I) . 
\end{align*}
For $a\in \R$, we will simply write 
\begin{align*}
P_{<a}(t)&:=P_{(-\infty,a)}(t) , & 
\mH_{<a}(t) &:= \Ran(P_{<a}) , 
\end{align*}
and similarly for $\geq a$. 
\end{df}

\begin{df}[{\cite{Phi96}}]
A partition 
$$0=t_0<t_1<...<t_N=T$$
together with numbers $a_n\in \Rp$ for $1\leq n\leq N$ will be called a \emph{flow partition} (for $A$), if 
for each $n$ and $t\in [t_{n-1},t_n]$ we have $a_n\notin\spec(A(t))$ and $\mH_{[0,a_n)}(t)$ is finite dimensional.
For such a partition, the spectral flow is defined as 
$$
\sfl(A)=\sum\limits_{n=1}^N \Dim(\mH_{[0,a_n)}(t_n))-\Dim(\mH_{[0,a_n)}(t_{n-1})) . 
$$
\end{df}
The spectral flow is well-defined, i.e.\ a flow partition exists and the spectral flow is independent of the choice of flow partition (\cite{Phi96}).
We also note that the spectral flow is unchanged by conjugating with unitaries, as this does not change the dimensions of the spectral subspaces.

A pair $(P,Q)$ of projections on $\mH$ is called a \emph{Fredholm pair}, if the restricted operator $Q|_{{\Ran(P)}\rightarrow{\Ran(Q)}}$ is Fredholm. 
In this case the \emph{(relative) index of $(P,Q)$} is defined to be the Fredholm index of $Q|_{{\Ran(P)}\rightarrow{\Ran(Q)}}$. 
If $P-Q$ is a compact operator, then $(P,Q)$ is a Fredholm pair. For more details regarding the index of a pair of projections, we refer to \cite{ASS94}.
We also quote the following result, which states that continuous families of Fredholm pairs have constant index:
\begin{lemma}[{\cite[Lemma 3.2]{Les05}}]
\label{contfred}
If $P,Q\colon [0,1]\rightarrow B(\mH)$ are continuous paths of projections in some Hilbert space $\mH$, such that $(P(t),Q(t))$ is a Fredholm pair for all $t\in[0,1]$,  then 
$$\ind(P(0),Q(0))=\ind(P(1),Q(1)).$$
\end{lemma}

The following result relates the spectral flow of a family to the relative index of the spectral projections at the endpoints. 
Its proof combines arguments from \cite[\S4.2]{BS19} (reformulated in terms of spectral projections) with Lemma \ref{contfred}. 
\begin{thm}
\label{flowind}
If $(P_{<0}(0),P_{<0}(t))$ is a Fredholm pair for all $t\in[0,T]$, we have
$$\sfl(A)=\ind(P_{<0}(0),P_{<0}(T)).$$
\end{thm}
\begin{proof}
Let $(t_n)$, $(a_n)$ be a flow partition for $A$. 
For any $a\in\R$, let $P_{<a}(t)_r$ denote the restriction 
$$P_{<a}(t)_r:=P_{<a}(t)\big|_{{\mH_{<0}(0)}\rightarrow{\mH_{<a}(t)}} .$$
Since $(P_{<0}(0),P_{<0}(t))$ is a Fredholm pair, we know that $P_{<0}(t)_r$ is Fredholm. 
Fix $n\leq N$. For $t\in[t_{n-1},t_n]$, let 
$$P_t:=P_{<0}(t)\big|_{{\mH_{<a_n}(t)}\rightarrow{\mH_{<0}(t)}}$$
be the restriction of $P_{<0}(t)$, which is Fredholm since $\mH_{[0,a_n)}(t)$ is finite-dimensional.  
We have
$$P_{<0}(t)_r=P_tP_{<a_n}(t)_r.$$
As $P_{<0}(t)_r$ and $P_t$ are Fredholm,
it follows that $P_{<a_n}(t)_r$ is Fredholm as well, and we have
$$\ind(P_{<0}(t)_r)=\ind(P_t)+\ind(P_{<a_n}(t)_r)=\Dim(\mH_{[0,a_n)}(t))+\ind(P_{<a_n}(t)_r).$$
Since $a_n\notin\spec(A(t))$ for $t\in[t_{n-1},t_n]$, it follows from \cite[Ch.\ 6, Theorem 5.12]{Kato80} that $P_{<a_n}(t)_r$ is continuous in $t$ on $[t_{n-1},t_n]$. By Lemma \ref{contfred}, $(P_{<0}(0), P_{<a_n}(t))$ has constant index for $t\in [t_{n-1},t_n]$. Thus we have
\begin{align*}
\ind(P_{<a_n}(t_n)_r)=\ind(P_{<0}(0),P_{<a_n}(t_n))=\ind(P_{<0}(0),P_{<a_n}(t_{n-1}))=\ind(P_{<a_n}(t_{n-1})_r) . 
\end{align*}
Moreover, as $P_{<0}(0)_r$ is the identity on $\mH_{(-\infty,0)}(0)$, it has index $0$.
Combining everything, we get:
\begin{align*}
\ind(P_{<0}(0),P_{<0}(T))
&=\ind(P_{<0}(T)_r)\\
&=\ind(P_{<0}(T)_r)-\ind(P_{<0}(0)_r)\\
&=\sum\limits_{n=1}^N \ind( P_{<0}(t_n)_r)-\ind(P_{<0}(t_{n-1})_r)\\
&=\sum\limits_{n=1}^N \Dim(\mH_{[0,a_n)}(t_n))+\ind(P_{<a_n}(t_n)_r)\\
&\qquad-\Dim(\mH_{[0,a_n)}(t_{n-1}))-\ind(P_{<a_n}(t_{n-1})_r)\\
&=\sum\limits_{n=1}^N \Dim(\mH_{[0,a_n)}(t_n))-\Dim(\mH_{[0,a_n)}(t_{n-1}))\\
&=\sfl(A).
\qedhere
\end{align*}
\end{proof}

\begin{remark}
\label{remark:flowind}
A similar theorem was proven in \cite[Theorem 3.6]{Les05}. There, the family $A$ is only assumed to `Riesz continuous' (instead of norm-continuous). On the other hand, \cite[Theorem 3.6]{Les05} makes the additional assumption that the difference $A(t)-A(0)$ is relatively compact (with respect to $A(0)$). The latter assumption ensures (by \cite[Corollary 3.5]{Les05}) that $P_{<0}(0)-P_{<0}(t)$ is compact, so in particular $(P_{<0}(0),P_{<0}(t))$ is a Fredholm pair for all $t\in[0,T]$. 
Thus, in the case of \emph{norm}-continuous families, our Theorem \ref{flowind} generalises \cite[Theorem 3.6]{Les05}, since we do not require compactness of $P_{<0}(0)-P_{<0}(t)$. 
\end{remark}

\section{The `Riemannian' APS-index}
\label{Riem}

In this section, we slightly strengthen Assumption \ref{set_sf} by assuming that $A(t)$ is not only Fredholm but in fact has compact resolvents. 
Thus throughout this section we consider the following setting. 
\begin{assumption}
\label{set_Riem_APS}
Let $\mH$ be a separable Hilbert space, let $W\subseteq \mH$ be a dense subspace such that the inclusion is compact, and let $\famop$ be a family of unbounded self-adjoint operators on $\mH$ with constant domain $W$. 
We equip $W$ with the graph norm of $A(0)$. 
We \emph{assume} that the family $A\colon[0,T]\to B(W,\mH)$ is norm-continuous.
\end{assumption}

We continuously extend the family $\famop$ to a family $\{\tilde A(t)\}_{t\in\R}$ parametrised by the whole real line, defined by 
\[
\tilde A(t) := 
\begin{cases}
A(0) , & \text{if } t\leq0 , \\
A(t) , & \text{if } 0\leq t\leq T , \\
A(T) , & \text{if } t\geq T .
\end{cases}
\]
We introduce the following spaces:
\begin{align*}
\mW &:= L^2(\R,W) \cap H^1(\R,\mH) , \\
\mW_{\APS} &:= \big\{ f\in L^2([0,T],W) \cap H^1([0,T],\mH) : f(0)\in\mH_{<0}(0) , \; f(T)\in\mH_{\geq0}(T) \big\} , \\
\mW_{\APS}^\dagger &:= \big\{ f\in L^2([0,T],W) \cap H^1([0,T],\mH) : f(0)\in\mH_{\geq0}(0) , \; f(T)\in\mH_{<0}(T) \big\} . 
\end{align*}
Here $H^1(\R)\subset L^2(\R)$ denotes the standard first Sobolev space, and $H^1(\R,\mH) \simeq H^1(\R)\otimes\mH$. We note that the evaluation $\ev_t\colon H^1(\R,\mH)\to\mH$, $f\mapsto f(t)$, is well-defined (since elements in $H^1(\R)$ are continuous). 

\begin{df}
\label{df:Riem_D}
We consider the following operators:
\begin{itemize}
\item $\tilde D := \partial_t + \tilde A$ on the Hilbert space $L^2(\R,\mH)$ with initial domain $C_c^1(\R,W)$, and 
\[
\tilde\D := \mat{0}{-\partial_t+\tilde A}{\partial_t+\tilde A}{0} 
\]
on the Hilbert space $L^2(\R,\mH)^{\oplus2}$ with initial domain $C_c^1(\R,W)^{\oplus2}$. 
\item $D_{\APS} := \partial_t + A$ on the Hilbert space $L^2([0,T],\mH)$ with initial domain 
\[
\Dom D_{\APS} := \big\{ f\in C^1([0,T],W) : f(0)\in\mH_{<0}(0) , \; f(T)\in\mH_{\geq0}(T) \big\} , 
\]
and 
\[
\D_{\APS} := \mat{0}{-\partial_t+A}{\partial_t+A}{0} 
\]
on the Hilbert space $L^2([0,T],\mH)^{\oplus2}$ with initial domain $\Dom D_{\APS} \oplus \Dom D_{\APS}^\dagger$, where  
\[
\Dom D_{\APS}^\dagger := \big\{ f\in C^1([0,T],W) : f(0)\in\mH_{\geq0}(0) , \; f(T)\in\mH_{<0}(T) \big\} . 
\]
\end{itemize}
\end{df}

\begin{prop}
\label{prop:adjoint_APS}
\begin{enumerate}
\item The closure of the operator $\tilde\D$ is self-adjoint on the domain $\mW^{\oplus2}$, and for any $f\in C_c^\infty(\R)$, the operators $f\cdot(\tilde\D\pm i)^{-1}$ on $L^2(\R,\mH)$ are compact. 
\item The closure of the operator $\D_{\APS}$ is self-adjoint on the domain $\mW_{\APS}\oplus\mW_{\APS}^\dagger$, and the operators $(\D_{\APS}\pm i)^{-1}$ on $L^2([0,T],\mH)$ are compact. 
In particular, $\D_{\APS}$ is Fredholm. 
\end{enumerate}
\end{prop}
\begin{proof}
The self-adjointness of $\tilde\D$ on $\mW^{\oplus2}$ follows as in \cite[Proposition 3.16]{vdD19_Index_DS}. Moreover, we know from \cite[Proposition 4.1]{vdD19_Index_DS} (cf.\ the proof of \cite[Theorem 6.7]{KL13}) that $f\cdot (\tilde\D\pm i)^{-1}$ is compact for every $f\in C_0(M)$, which proves (1). 

Next, we will prove the self-adjointness of $\D_{\APS}$. 
Since $A$ is norm-continuous, we can pick $0<\varepsilon<\frac12$ small enough such that 
\begin{align*}
\sup_{t\in[0,\varepsilon]} \big\| \big( A(t) - A(0) \big) \big( A(0)-i \big)^{-1} \big\| &< \frac12 , & 
\sup_{t\in[T-\varepsilon,T]} \big\| \big( A(t) - A(T) \big) \big( A(0)-i \big)^{-1} \big\| &< \frac12 . 
\end{align*}
We consider a new norm-continuous family $A_L\colon[0,\infty)\to B(W,\mH)$ given by
\[
A_L(t) := 
\begin{cases}
A(t) , & \text{ if } 0\leq t\leq\varepsilon , \\
A(\varepsilon) , & \text{ if } t\geq\varepsilon . 
\end{cases}
\]
Consider the operators 
\begin{align*}
\D_0 &:= \mat{0}{-\partial_t+A(0)}{\partial_t+A(0)}{0} , & 
\D_L &:= \mat{0}{-\partial_t+A_L}{\partial_t+A_L}{0} ,
\end{align*}
on the Hilbert space $L^2([0,\infty),\mH)$ with domain $\mW_L\oplus\mW_L^\dagger$, where we introduce the spaces 
\begin{align*}
\mW_L &:= \big\{ f\in L^2([0,\infty),W) \cap H^1([0,\infty),\mH) : f(0)\in\mH_{<0}(0) \big\} , \\
\mW_L^\dagger &:= \big\{ f\in L^2([0,\infty),W) \cap H^1([0,\infty),\mH) : f(0)\in\mH_{\geq0}(0) \big\} . 
\end{align*}
We recall that the operator $\D_0$ is self-adjoint (see \cite[Proposition 2.12]{APS75} or, for the more abstract setting, \cite[Corollary 4.6]{BL01} and \cite[Proposition 4.11]{CPR10}). 
As in the proof of \cite[Lemma 3.13]{vdD19_Index_DS}, we can estimate 
\begin{align*}
\big\| ( \D_L - \D_0 ) (\D_0-i)^{-1} \big\| 
&\leq \big\| ( A_L - A(0) ) (A(0)-i)^{-1} \big\| \; \big\| ( A(0)-i ) (\D_0-i)^{-1} \big\| \\
&\leq \sup_{t\in[0,\varepsilon]} \big\| \big( A(t) - A(0) \big) \big( A(0)-i \big)^{-1} \big\| < \frac12 , 
\end{align*}
where we have used that $\big\| ( A(0)-i ) (\D_0-i)^{-1} \big\| \leq 1$. 
By the Kato-Rellich Theorem, it then follows that $\D_L$ is also self-adjoint on the domain $\mW_L\oplus\mW_L^\dagger$. 
Similarly, the operator 
\begin{align*}
\D_R &:= \mat{0}{-\partial_t+A_R}{\partial_t+A_R}{0} , & 
A_R(t) &:= 
\begin{cases}
A(T-\varepsilon) , & \text{ if } t\leq T-\varepsilon , \\
A(t) , & \text{ if } T-\varepsilon\leq t\leq T , 
\end{cases}
\end{align*}
is self-adjoint on the domain $\mW_R\oplus\mW_R^\dagger$, where 
\begin{align*}
\mW_R &:= \big\{ f\in L^2((-\infty,T],W) \cap H^1((-\infty,T],\mH) : f(T)\in\mH_{\geq0}(T) \big\} , \\
\mW_R^\dagger &:= \big\{ f\in L^2((-\infty,T],W) \cap H^1((-\infty,T],\mH) : f(T)\in\mH_{<0}(T) \big\} . 
\end{align*}
Now pick smooth functions 
$\chi_L,\chi_I,\chi_R\colon\R\to[0,1]$ 
such that $\{\chi_L^2,\chi_I^2,\chi_R^2\}$ is a partition of unity subordinate to the open cover $\{(-\infty,\varepsilon),(0,T),(T-\varepsilon,\infty)\}$ of $\R$. 
For $\lambda>0$, we define 
\[
R_\pm(\lambda) := \chi_L (\D_L\pm i\lambda)^{-1} \chi_L + \chi_I (\tilde\D\pm i\lambda)^{-1} \chi_I + \chi_R (\D_R\pm i\lambda)^{-1} \chi_R .
\]
Since $\D_{\APS}$ agrees with $\D_L$ on $[0,\varepsilon)$, agrees with $\tilde\D$ on $(0,T)$, and agrees with $\D_R$ on $(T-\varepsilon,T]$, we note that $\Ran R_\pm(\lambda)\subset\mW_{\APS}\oplus\mW_{\APS}^\dagger$, and we can compute 
\begin{gather*}
(\D_{\APS}\pm i\lambda) R_\pm(\lambda) = \Id + K_\pm(\lambda) , \\
K_\pm(\lambda) := [\D_L,\chi_L](\D_L\pm i\lambda)^{-1} \chi_L + [\tilde\D,\chi_I](\tilde\D\pm i\lambda)^{-1} \chi_I + [\D_R,\chi_R](\D_R\pm i\lambda)^{-1} \chi_R .
\end{gather*}
By choosing $\lambda$ large enough, we may ensure that $\|K_\pm(\lambda)\|<1$, so that $\Id+K_\pm(\lambda)$ is invertible, and then $R_\pm(\lambda)\big(\Id+K_\pm(\lambda)\big)^{-1}$ is a right inverse for $\D_{\APS}\pm i\lambda$. Similarly, we can also construct a left inverse for $\D_{\APS}\pm i\lambda$. Thus $\D_{\APS}\pm i\lambda$ is invertible, which proves that $\D_{\APS}$ is self-adjoint. 

Finally, we know from (1) that $\chi_I (\tilde\D\pm i\lambda)^{-1}$ is compact. 
Furthermore, the operator $\chi_L (\D_0\pm i\lambda)^{-1}$ is compact by \cite[Proposition 4.14]{CPR10}, and since $\Dom\D_L=\Dom\D_0$ this implies that $\chi_L (\D_L\pm i\lambda)^{-1}$ is compact. 
Similarly, also $\chi_R (\D_R\pm i\lambda)^{-1}$ is compact. 
Hence also $R_\pm(\lambda)$ is compact, and therefore $(\D_{\APS}\pm i)^{-1}$ is compact. This completes the proof of (2). 
\end{proof}

\subsection{APS-index and spectral flow}

We first consider the special case where the family $A$ is invertible at the endpoints of the interval $[0,T]$. 
In this case, we recall the following equality between index and spectral flow on the real line. 
\begin{thm}[{\cite[Theorem 2.1]{AW11}}]
\label{thm:Riem_index_sfl}
If $A(0)$ and $A(T)$ are invertible, then the operator $\tilde D$ is Fredholm, and we have the equality
\[
\ind(\tilde D) = \sfl(A) . 
\]
\end{thm}

\begin{prop}
\label{prop:Riem_APS_index}
Assume that $A(0)$ and $A(T)$ are invertible. Then we have isomorphisms 
\begin{align*}
\Ker D_{\APS} &\simeq \Ker \tilde D , & 
\Ker {D_{\APS}}^* &\simeq \Ker\tilde D^* , 
\end{align*}
and consequently we have the equality 
\[
\ind(D_{\APS}) = \ind(\tilde D) . 
\]
\end{prop}
\begin{proof}
The proof is an adaptation of the argument in \cite[Proposition 3.11]{APS75}.
Let $\{\psi_\lambda(t)\}_{\lambda\in\spec(A(t))}$ be an orthonormal basis of $\mH$ consisting of eigenvectors $\psi_\lambda(t)$ of $A(t)$ with eigenvalue $\lambda$ (where the eigenvalues are counted with multiplicities). 
For any element $f\in\Ker D_{\APS}$, we can write $f(0) = \sum_\lambda \mu_\lambda\psi_\lambda(0)$, for some $\mu_\lambda\in\C $ (recall that the evaluation $\ev_t\colon\Dom D_{\APS}\to\mH$ is well-defined, since $\Dom D_{\APS}\subset H^1(\R,\mH)$).  
We will extend $f$ to an element $\tilde f\in\Ker\tilde D$, as follows. 
Solving $(\partial_t+\tilde A)\tilde f=0$ for $t<0$ yields \[\frac{\partial}{\partial t} \<\psi_\lambda(0),\tilde f(t)\> = - \lambda \<\psi_\lambda(0),\tilde f(t)\>,\]
which implies
\[
\tilde f(t) = \sum_{\lambda<0} e^{-\lambda t} \mu_\lambda \psi_\lambda(0) , \qquad t\leq0 . 
\]
Here we have used the APS boundary condition $f(0) \in \mH_{<0}(0)$, which tells us that $\mu_\lambda = 0$ whenever $\lambda\geq0$. 
Writing instead $f(T) = \sum_\lambda \nu_\lambda \psi_\lambda(T)$ and solving $(\partial_t+\tilde A)\tilde f=0$ for $t>T$, we similarly obtain
\[
\tilde f(t) = \sum_{\lambda>0} e^{-\lambda(t-T)} \nu_\lambda \psi_\lambda(T) , \qquad t\geq T , 
\]
where we have used that $A(T)$ is invertible, so that $\lambda\neq0$. We can then define a map $\iota\colon\Ker D_{\APS}\to\Ker\tilde D$ by defining
\[
\iota(f)(t) := 
\begin{cases}
\sum_{\lambda<0} e^{-\lambda t} \mu_\lambda \psi_\lambda(0) , & \text{if } t\leq0 , \\
f(t) , & \text{if } 0\leq t\leq T , \\
\sum_{\lambda>0} e^{-\lambda(t-T)} \nu_\lambda \psi_\lambda(T) , & \text{if } t\geq T . 
\end{cases}
\]
This map $\iota$ is clearly injective. Conversely, given any $\xi\in\Ker\tilde D$, the requirement that $\xi$ is square-integrable ensures that $\xi$ must have the above form on $(\infty,0]$ and on $[T,\infty)$. By continuity, this implies that $\xi|_{[0,T]}$ satisfies the boundary conditions $\xi(0)\in\mH_{<0}(0)$ and $\xi(T)\in\mH_{>0}(T)$, and we conclude that $\xi = \iota(\xi|_{[0,T]})$. 
Thus we have shown that $\iota$ yields an isomorphism $\Ker D_{\APS} \xrightarrow{\simeq} \Ker\tilde D$. 
Similarly, we also obtain an isomorphism $\bar\iota\colon \Ker D_{\APS}^* \xrightarrow{\simeq} \Ker\tilde D^*$ given by  
\[
\bar\iota(f)(t) := 
\begin{cases}
\sum_{\lambda>0} e^{\lambda t} \mu_\lambda \psi_\lambda(0) , & \text{if } t\leq0 , \\
f(t) , & \text{if } 0\leq t\leq T , \\
\sum_{\lambda<0} e^{\lambda(t-T)} \nu_\lambda \psi_\lambda(T) , & \text{if } t\geq T . 
\end{cases}
\]
Since we know from Proposition \ref{prop:adjoint_APS} that $D_{\APS}$ is Fredholm, and from Theorem \ref{thm:Riem_index_sfl} that $\tilde D$ is Fredholm, the final statement follows immediately. 
\end{proof}

Proposition \ref{prop:Riem_APS_index} and Theorem \ref{thm:Riem_index_sfl} then immediately yield:
\begin{coro}
\label{coro:Riem_APS_index_sfl_inv}
If $A(0)$ and $A(T)$ are invertible, then 
\[
\ind(D_{\APS}) = \sfl(A) . 
\]
\end{coro}

Next, we will prove the equality $\ind(D_{\APS}) = \sfl(A)$ in general, by reducing to the special case with invertible endpoints, as follows. 
\begin{df}
Consider a smooth function $\chi\colon\R\to[0,1]$ such that $\chi\equiv1$ near $0$ and $\supp\chi\subset(-\varepsilon,\varepsilon)$ for some $\varepsilon<\frac12$. 
We define a family of compact operators $\{K(t)\}_{t\in\R}$ on $\mH$ by 
\[
K(t) := \chi(t) P_0(A(0))+\chi(T-t) P_0(A(T))
\]
Here $P_0(A(t))$ denotes the projection onto the kernel of $A(t)$. 
We then obtain a new family $\{B(t)\}_{t\in[0,T]}$ of unbounded self-adjoint operators on $\mH$ with constant domain $W$, given by 
\[
B(t) := A(t) + K(t) , \qquad t\in[0,T] . 
\]
We note that the family $\{B(t)\}_{t\in[0,T]}$ is again norm-continuous, and therefore satisfies Assumption \ref{set_Riem_APS}.
As above, we continuously extend $\{B(t)\}_{t\in[0,T]}$ to a family $\{\tilde B(t)\}_{t\in\R}$ on the real line. 
As in Definition \ref{df:Riem_D} and Proposition \ref{prop:adjoint_APS}, we then define the operators 
\begin{align*}
\tilde D' &:= \partial_t + \tilde B , \qquad \text{on } \Dom\tilde D' := \mW^{\oplus2} , \\
D'_{\APS} &:= \partial_t + B , \qquad \text{on } \Dom D'_{\APS} := \mW_{\APS}\oplus\mW_{\APS}^\dagger . 
\end{align*}
\end{df} 

Let us make a few observations. First of all, the family $\{K(t)\}_{t\in[0,T]}$ is chosen such that the operators $B(0)$ and $B(T)$ are invertible. 
Second, we note that, in our conventions of both the spectral flow and the APS boundary conditions, zero belongs to the positive spectrum. Since the operators $K(0)$ and $K(T)$ move the kernels of $A(0)$ and $A(T)$ (respectively) into the strictly positive spectrum of $B(0)$ and $B(T)$ (respectively), we have $P_{\geq0}(B(0)) = P_{\geq0}(A(0))$ and $P_{\geq0}(B(T)) = P_{\geq0}(A(T))$. 
Consequently, we find that replacing $A$ by $B$ does not affect the APS boundary conditions, and we have the equality 
\begin{align*}
\Dom D'_{\APS} &= \Dom D_{\APS} . 
\end{align*}

\begin{lemma}
\label{lem:B}
We have the equalities 
\begin{align*}
\sfl(B) &= \sfl(A) , & 
\ind(D'_{\APS}) &= \ind(D_{\APS}) . 
\end{align*}
\end{lemma}
\begin{proof}
We first prove the equality $\sfl(B) = \sfl(A)$. 
Since $K(t)$ is a family of compact operators, we can consider the straight-line homotopy $B_s := A + s K = \{ A(t) + s K(t) \}_{t\in[0,1]}$ for $s\in[0,1]$. 
It then follows from \cite[Cor.\ 3.4]{SW19} that $\sfl(B) = \sfl(A)$, if the spectral flows $\sfl\big( \{B_s(0)\}_{s\in[0,1]} \big)$ and $\sfl\big( \{B_s(T)\}_{s\in[0,1]} \big)$ are both identically zero. That the latter condition is satisfied can be 
checked directly, using that the spectral projections $P_{\geq0}(B_s(0))$ and $P_{\geq0}(B_s(T))$ are constant. 

Regarding the second equality, we recall from Proposition \ref{prop:adjoint_APS} that $D_{\APS}$ and $D'_{\APS}$ are Fredholm. 
We have already seen that $D'_{\APS}$ and $D_{\APS}$ have the same APS boundary conditions and therefore the same domain. 
Since the difference $D'_{\APS}-D_{\APS}$ is bounded and $D_{\APS}$ has compact resolvents by Proposition \ref{prop:adjoint_APS}, we see that $D'_{\APS}$ is a relatively compact perturbation of $D_{\APS}$, and therefore the index is the same. 
\end{proof}

\begin{thm}
\label{thm:Riem_APS_index_sfl}
We have the equality
\[
\ind(D_{\APS}) = \sfl(A) . 
\]
\end{thm}
\begin{proof}
Combining the equalities from Lemma \ref{lem:B} with Corollary \ref{coro:Riem_APS_index_sfl_inv}, we obtain the sequence of equalities 
\[
\ind(D_{\APS}) = \ind(D'_{\APS}) = \sfl(B) = \sfl(A) . 
\qedhere
\]
\end{proof}

\section{The `Lorentzian' APS-index}
\label{sec:APS-ind}

In this section, we strengthen Assumption \ref{set_sf} by assuming that $A(t)$ is not only norm-continuous but in fact is strongly continuously differentiable. Thus throughout this section we consider the following setting. 
\begin{assumption}
\label{set}
Let $\mH$ be a separable Hilbert space, let $W\subseteq \mH$ be a dense subspace, and let $\famop$ be a family of unbounded self-adjoint Fredholm operators on $\mH$ with constant domain $W$. 
We equip $W$ with the graph norm of $A(0)$. 
We \emph{assume} that the family $A\colon[0,T]\to B(W,\mH)$ is strongly continuously differentiable.
\end{assumption}

\begin{df}
For $s<t\in [0,T]$, let $D|_{[s,t]}$ denote the closure in $L^2([s,t],\mH)$ of
$$\ddt-iA$$
with initial domain $C^1([s,t],W)$.
Define $D:=D|_{[0,T]}$.
\end{df}

\subsection{The Evolution Operator}
\label{sec:evol}

\begin{thm}[{\cite[Ch.\ 5]{Pazy83}}]
\label{evop}
There is a family of bounded operators $Q(t,s)\colon \mH\rightarrow \mH$ for $s,t\in[0,T]$, satisfying the following conditions (for all $s,t,r\in [0,T]$):
\begin{enumerate}
\item \label{e1} $Q(s,s)=\Id$; 
\item \label{e2} $Q(t,s)Q(s,r)=Q(t,r)$; 
\item \label{e3} $Q(t,s)$ is an isometry (of $\mH$); 
\item \label{e4} $Q(t,s)(W)\subseteq W$ and $Q(t,s)\colon W\rightarrow W$ is bounded;
\item \label{e5} $Q$ is strongly continuously differentiable in $B(W,\mH)$ with derivatives 
$$\pdt Q(t,s)=iA(t)Q(t,s)$$ and 
$$\pds Q(t,s)=-Q(t,s)iA(s).$$
\item \label{e6} $Q(t,s)x$ (as a function of $s$ and $t$) is continuous in $\mH$ for $x\in \mH$ and continuous in $W$ for $x\in W$.
\end{enumerate}
\end{thm}
\begin{proof}
Most of the statement is proven in \cite[Ch.\ 5]{Pazy83} for a more general situation (without assuming $A(t)$ to be self-adjoint).
To be precise, \cite[Ch.\ 5, Theorem 4.8]{Pazy83} provides the operator $Q(t,s)$ for $t\geq s$, 
satisfying for all $t\geq s\geq r$ the conditions \ref{e1}, \ref{e2}, \ref{e4}, and \ref{e6} (for \ref{e4} we note that the boundedness of $Q(t,s)\colon W\rightarrow W$ follows from the inclusion $Q(t,s)(W)\subseteq W$ and the closed graph theorem), as well as the equalities $\pdt^+ Q(t,s)=iA(t)Q(t,s)$ and $\pds Q(t,s)=-Q(t,s)iA(s)$. 
For $x\in W$, the calculation
$$\pdt^+ ||Q(t,s)x||^2=2\Real(\<Q(t,s)x,iA(t)Q(t,s)x\>)=0$$
together with $Q(s,s)=\Id$ shows that $Q(t,s)$ is an isometry, so in fact \ref{e3} is also satisfied. 

Similarly, for $s\geq t$, we obtain the operator $Q'(t,s)$ associated to the family $-A(T-\cdot)$. 
Then the operator $Q(t,s):=Q'(T-t,T-s)$ satisfies the same conditions for all $r\geq s\geq t$. 
Since both definitions agree at $t=s$, we get a strongly continuous family $Q(t,s)$ for all $t$ and $s$.
For $s\leq t$, we compute (using Proposition \ref{prop:diff})
$$\pdt^+Q(s,t)Q(t,s)=-Q(s,t)iA(T-(T-t))Q(t,s)+Q(s,t)iA(t)Q(t,s)=0.$$
Thus $Q(s,t)$ and $Q(t,s)$ are mutually inverse (as this holds at $t=s$), and we find that \ref{e2} is in fact satisfied for arbitrary $s,t,r$.
Finally, as 
$$\pdt^\pm Q(t,s)=\left.\frac{\partial}{\partial r}^\pm Q(r,t)\right|_{r=t}Q(t,s)=iA(t)Q(t,s),$$
we get the $t$-derivatives in \ref{e5}, and we note that $Q(t,s)$ is strongly continuously differentiable in $B(W,\mH)$ because $A(t)Q(t,s)$ is strongly continuous in $B(W,\mH)$. 
\end{proof}

We will refer to $Q$ as the \emph{evolution operator}. The unitary operator $Q(t,s)$ can be thought of as evolving the initial data at time $s$ to the final data at time $t$, subject to the equation $Df=0$. More precisely, the function $f(t):=Q(t,s)x$ is the unique solution to the equations
\begin{align*}
Df&=0, \qquad f(s)=x, 
\end{align*}
When replacing $Df=0$ with $Df=g$ for some $g\in\LH$, the equations still have a unique solution:
\begin{thm}[Well-posedness of the Cauchy problem]
\label{thm:Cauchy}
The domain $\Dom(D)$ is a subspace of $C([0,T],\mH)$ (with maximum norm) with bounded inclusion. 
For all $t\in [0,T]$ the map
$$D\oplus \ev_t \colon  \Dom(D)\rightarrow \LH \oplus \mH$$
is an isomorphism, where $\ev_t\colon C([0,T],\mH)\to\mH$ denotes evaluation at $t$. 
\end{thm}
\begin{proof}
For the first statement, let $f\in C^1([0,T],W)$. Using that $\Real\big(\<f(t),iA(t)f(t)\>\big)=0$, we have
\[
T\|{f(t)}\|^2-\|f\|_{L^2}^2 
= \int\limits_0^T \int\limits_s^t \frac{d}{dr} \|f(r)\|^2 dr ds 
= \int\limits_0^T \int\limits_s^t 2\Real\big(\<f(r),Df(r)\>\big) dr ds . 
\]
This allows us to estimate
$$
T\|{f(t)}\|^2-\|f\|_{L^2}^2
\leq \int\limits_0^T 2 \|f\|_{L^2} \|Df\|_{L^2} ds 
\leq T\|f\|_{D}^2 ,
$$
which ensures that the inclusion $C^1([0,T],W)\into C([0,T],\mH)$ extends to a bounded inclusion $\Dom(D)\into C([0,T],\mH)$.

The second statement follows by checking that the map $F_s\colon\LH\oplus\mH\to\Dom D$, given for $g\in\LH$ and $x\in\mH$ by 
$$F_s(g,x)(t):=Q(t,s)x+\int\limits_s^tQ(t,r)g(r)dr,$$
is an inverse for $D\oplus\ev_s$. 
Indeed, an explicit computation shows that $F_s\circ (D\oplus \ev_s)$ and $(D\oplus \ev_s)\circ F_s$ are the identity on $C^1([0,T],W)$ and $C^1([0,T],W)\oplus\mH$ respectively. For $(f,x)$ in the latter space, we can then estimate
\[
\|F_s(f,x)\|_{L^2}^2 
\leq \|x\|_{L^2}^2 + \left\| \int_s^{\cdot}\|f(r)\|dr\right\|_{L^2}^2 
\leq T \|x\|^2 + T \|f\|_{L^1}^2 
\leq T \|x\|^2 + C T \|f\|_{L^2}^2 ,
\]
for some $C>0$. It follows that 
\[
\|F_s(f,x)\|_D^2 = \|F_s(f,x)\|_{L^2}^2 + \|f\|_{L^2}^2 
\leq (1+T+CT) (\|f\|_{L^2}^2+\|x\|^2) . 
\]
Thus $F_s$ maps continuously into $\Dom(D)$, whence the two compositions are the identity everywhere.
\end{proof}
Using the above theorem, we can rewrite the evolution operator in a concise way that highlights its connection to the Cauchy problem: 
\begin{align}
Q(t,s)x=\ev_t \circ (D\oplus \ev_s)^{-1}(0,x).
\end{align}

\subsection{The APS-index and spectral projections}
\label{sec:APS-ind_proj}
In the following, we will use the splitting of $\mH$ in positive and negative spectral subspaces of $A(t)$, in order to define APS boundary conditions. 
For any $t\in[0,T]$, we consider (as before) the spectral projections 
\begin{align*} 
P_{<0}(t)&:=P_{(-\infty,0)}(t) , & 
P_{\geq 0}(t)&:=P_{[0,\infty)}(t)=\Id-P_{<0}(t),
\end{align*}
and the corresponding subspaces 
\begin{align*}
\mH_{<0}(t)&:=\Ran(P_{<0}(t)) , & 
\mH_{\geq 0}(t)&:=\Ran(P_{\geq 0}(t)) . 
\end{align*}

\begin{df}
For $s<t\in[0,T]$, let $(D|_{[s,t]})_{\APS}$ be the restriction of $D|_{[s,t]}$ to the domain 
$$\Dom\big((D|_{[s,t]})_{\APS}\big):=\big\{f\in \Dom(D) : f(s)\in \mH_{<0}(s),f(t)\in \mH_{\geq 0}(t)\big\}.$$
We will write 
$$D_{\APS}:=(D|_{[0,T]})_{\APS}.$$
\end{df}

We will relate the index of $D_{\APS}$ to the index of a pair of spectral projections. For this purpose, we consider the \emph{evolved spectral projections} defined as 
$$\hat P_{<a}(t):=Q(0,t)P_{<a}(t)Q(t,0).$$
Let $\hat P_{<a}(t)_r$ be the restriction of $\hat P_{<a}(t)$ to $\mH_{<0}(0)$ with codomain $Q(0,t)\mH_{<a}(t)$: 
$$\hat P_{<a}(t)_r:=\hat P_{<a}(t)\big|_{{\mH_{<0}(0)}\rightarrow{Q(0,t)\mH_{<a}(t)}}.$$
We note that $\hat P_{<a}(t)$ is the projection onto $Q(0,t)\mH_{<a}(t)$, and that (by construction) $\hat P_{<a}(t)_r$ is Fredholm with index $k$ if and only if the pair $(P_{<0}(0),\hat P_{<a}(t))$ is Fredholm with index $k$. 
The following result is partly based on the arguments from \cite[\S3]{BS19}. 
\begin{thm}
\label{inD}
$D_{\APS}$ and $\hat P_{<0}(T)_r$ have isomorphic kernel and cokernel. In particular, $D_{\APS}$ is Fredholm with index $k$ if and only if $(P_{<0}(0),\hat P_{<0}(T))$ is a Fredholm pair with index $k$.
\end{thm}
\begin{remark}
By replacing $A$ by $A|_{[0,t]}$, we obtain for any $t\in[0,T]$ that $(D|_{[0,t]})_{\APS}$ is Fredholm with index $k$ if and only if $(P_{<0}(0),\hat P_{<0}(t))$ is Fredholm with index $k$.
\end{remark}
\begin{proof}
We have
\begin{align*}
\Ker(D_{\APS})&=\{f\in \Dom(D) : Df=0,f(0)\in \mH_{<0}(0), f(T)\in \mH_{\geq 0}(T)\}\\
&\cong\{f(0)\in \mH_{<0}(0) : Q(T,0)f(0)\in \mH_{\geq 0}(T)\}\\
&=\mH_{<0}(0)\cap Q(0,T)\mH_{\geq 0}(T)\\
&=\Ker(\hat P_{<0}(T)_r),
\end{align*}
where in the second line, we use that $Df=0$ implies $f(t)=Q(t,0)f(0)$, so $f\mapsto f(0)$ is an isomorphism.

For $g\in \LH$ define
$$E(g):=\ev_T \circ (D\oplus \ev_0)^{-1}(g,0).$$
Note that 
$$\ev_T \circ (D\oplus \ev_0)^{-1}(g,x)=E(g)+Q(T,0)x.$$

We will first show that $E\colon \LH \rightarrow \mH$ is surjective. 
Thus, we need to show that functions in $\Dom(D)$ that vanish at $0$ can take any value at $T$.  For $z\in \mH$ choose $f\in \Dom(D)$ with $f(T)=z$ (a possible choice is  $f(t)=Q(t,T)z$) and let
$\phi(t):=\frac{t}{T}$. 
Since multiplication with $\phi$ preserves $\Dom(D)$, 
we  have $\phi f \in \Dom(D)$, with $\phi(0)f(0)=0$ and $\phi(T)f(T)=z$. We get 
$$E(D(\phi f))=\ev_T\circ(D\oplus \ev_0)^{-1}(D(\phi f),0)=\ev_T(\phi f)=z.$$
As $z$ was arbitrary, $E$ is surjective.

To determine the cokernel of $D_{\APS}$, we need to characterise its range. 
For $g\in \LH$, we have the following chain of equivalences:
\begin{align*}
g\in \Ran(D_{\APS})&\Leftrightarrow \exists f\in \Dom(D): f(0)\in \mH_{<0}(0)\wedge f(T) \in \mH_{\geq 0}(T) \wedge Df=g\\
&\Leftrightarrow \exists f(0)\in \mH_{<0}(0): \ev_T(D\oplus \ev_0)^{-1}(g,f(0))\in \mH_{\geq 0}(T)\\
&\Leftrightarrow\exists f(0)\in \mH_{<0}(0):\exists z\in \mH_{\geq 0}(T): E(g)+Q(T,0)(f(0))=z\\
&\Leftrightarrow\exists x\in Q(T,0)\mH_{<0}(0):\exists z\in \mH_{\geq 0}(T): E(g)=z-x\\
&\Leftrightarrow E(g)\in Q(T,0)\mH_{<0}(0)+ \mH_{\geq 0}(T)
\end{align*}
Defining 
$$V:=Q(T,0)\mH_{<0}(0)+ \mH_{\geq 0}(T)=P_{<0}(T)Q(T,0)\mH_{<0}(0)+ \mH_{\geq 0}(T),$$
(with the latter sum being orthogonal), we get
$$\Ran(D_{\APS})=\{g\in\LH : E(g)\in V\}=E^{-1}(V).$$
In particular, this also implies that $\Ker(E) = E^{-1}(\{0\}) \subset \Ran(D_{\APS})$. 
By the surjectivity of $E$, we therefore obtain the isomorphism 
\[
\LH/\Ran(D_{\APS} ) \cong \mH/V . 
\]
We can now conclude
\begin{align*}
\Coker(D_{\APS})&=\LH/\Ran(D_{\APS} )\\
&\cong \mH/V\\
&\cong \mH_{<0}(T) / \big(P_{<0}(T)Q(T,0)\mH_{<0}(0)\big)\\
&\cong (Q(0,T)\mH_{<0}(T)) / \big(Q(0,T)P_{<0}(T)Q(T,0)\mH_{<0}(0)\big)\\
&=\Coker(\hat P_{<0}(T)_r) . 
\qedhere 
\end{align*}
\end{proof}

\subsection{APS-index and spectral flow}
\label{sec:APS-ind_sf}

We recall that the strongly continuously differentiable family $\famop$ is norm-con\-tin\-u\-ous by Lemma \ref{lem:diff}.\ref{diffbound}, so in particular the results from Section \ref{sec:sf} apply. 
In order to combine Theorems \ref{flowind} and \ref{inD}, we need to consider a new `evolved' family $\hat A\colon[0,T]\to B(W,\mH)$ given by 
$$\hat A(t):=Q(0,t)A(t)Q(t,0).$$
For every $t\in[0,T]$, $\hat A(t)$ is self-adjoint and Fredholm, with domain $W$ (as $Q(t,0)^{-1}(W)=W$). 
As functional calculus is equivariant under conjugation with isometries,
we find that the spectral projections of $\hat A(t)$ correspond precisely to the evolved spectral projections from subsection \ref{sec:APS-ind_proj}:
$$\chi_{(-\infty,a)}(\hat A(t))=Q(0,t)\chi_{(-\infty,a)}(A(t))Q(t,0)=\hat P_{<a}(t).$$
Before we can apply Theorem \ref{flowind} to $\hat A$, we need to ensure that $\hat A$ is again norm-continuous, and we will prove that it is in fact strongly continuously differentiable in $B(W,\mH)$. 

\begin{lemma}
\label{hatdiff}
$\hat A\colon  J\rightarrow B(W,\mH)$ is strongly continuously differentiable with derivative
$$\hat A'(t) = Q(0,t)A'(t)Q(t,0) . $$
\end{lemma}
\begin{proof}
Let $R(t):=(A(t)-i)^{-1}$ for $t\in [0,T]$. $\hat A(t)$ is differentiable at $t$ if and only if 
$$\hat A(t)-i=Q(0,t)(A(t)-i)Q(t,0)$$ 
is. As $Q(t,0)$ and $Q(0,t)$ are strongly continuously differentiable in $B(W,\mH)$ and $R(t)$ is strongly continuously differentiable in $B(\mH,W)$, we get from Lemma \ref{lem:diff}.\ref{idiff} and Proposition \ref{prop:diff}.\ref{prodiff} that 
$$(\hat A(t)-i)^{-1}=Q(0,t)R(t)Q(t,0)$$
is strongly differentiable in $B(W,\mH)$. Its derivative is
\begin{align*}
&\ddt (\hat A(t)-i)^{-1} \\
&=\ddt Q(0,t)R(t)Q(t,0)\\
&=Q(0,t)R(t)iA(t)Q(t,0)-Q(0,t)R(t)A'(t)R(t)Q(t,0)-Q(0,t)iA(t)R(t)Q(t,0)\\
&=-Q(0,t)R(t)A'(t)R(t)Q(t,0).
\end{align*}
As this is strongly continuous in $B(\mH,W)$, Proposition \ref{prop:diff}.\ref{strictdiff} implies that $(\hat A(t)-i)^{-1}$ is strongly continuously differentiable in $B(\mH,W)$. By Lemma \ref{lem:diff}.\ref{idiff}, $\hat A(t)-i$ and hence $\hat A(t)$ are strongly continuously differentiable, with derivative
\begin{align*}
\hat A'(t)&=\ddt \big((\hat A(t)-i)^{-1}\big)^{-1}\\
&=-(\hat A(t)-i)\left(\ddt(\hat A(t)-i)^{-1}\right)(\hat A(t)-i)\\
&=Q(0,t)A'(t)Q(t,0) . 
\qedhere 
\end{align*}
\end{proof}

We now have all the pieces in place to prove our main result.
\begin{thm}
\label{main}
If $(D|_{[0,t]})_{\APS}$ is Fredholm for all $t\in[0,T]$, we have
$$\ind(D_{\APS})=\sfl(A).$$
\end{thm}
\begin{proof}
From Lemma \ref{hatdiff} we know that $\hat A$ satisfies Assumption \ref{set}. In particular, $\hat A$ is norm-continuous by Lemma \ref{lem:diff}.\ref{diffbound}, so we may apply  Theorem \ref{flowind}. The spectral projections of $\hat A$ are given by
$$\chi_{(-\infty, 0)}(\hat A(t))=\hat P_{<0}(t).$$
Using Theorem \ref{inD}, we know that $(\hat P_{<0}(0),\hat P_{<0}(t))$ is a Fredholm pair for all $t\in[0,T]$. 
Thus we obtain 
\begin{align*}
\ind (D_{\APS}) 
&\stackrel{\ref{inD}}{=}\ind(P_{<0}(0),\hat P_{<0}(T))
=\ind(\hat P_{<0}(0),\hat P_{<0}(T))
\stackrel{\ref{flowind}}{=}\sfl(\hat A)
=\sfl(A),
\end{align*}
where in the last step we used that the spectral flow is invariant under unitary conjugation. 
\end{proof}

\begin{example}
Consider the Lorentzian Dirac operator on a globally hyperbolic spacetime $M = \Sigma\times\R$, as studied in \cite{BS19} (and as described in the Introduction). 
It is shown in \cite[Lemma 2.6]{BS19}, using methods of Fourier integral operators, that the operator 
$$Q_{--}(t,0):=P_{<0}(t)Q(t,0)\big|_{{\mH_{<0}(0)}\rightarrow{\mH_{<0}(t)}}$$
is Fredholm for each $t\in[0,T]$. Since $\hat P_{<0}(t)_r = Q(0,t)Q_{--}(t,0)$ and $Q(0,t)$ is an invertible map between the codomains, it then follows that $\hat P_{<0}(t)_r$ is also Fredholm (and has the same index) for each $t\in T$. 
Using Theorems \ref{inD} and \ref{main}, we thus recover the equality\footnote{In \cite[\S4.1]{BS19}, there is actually an additional summand on the right hand side coming from the kernel of $A(T)$, due to a slightly different choice of boundary conditions.} $\ind(D_{\APS})=\sfl(A)$ from \cite[\S4.1]{BS19}.
\end{example}

It may be difficult to determine a priori whether $(D|_{[0,t]})_{\APS}$ is Fredholm for all $t\in[0,T]$. The following result provides a sufficient condition. 

\begin{prop}
If $A'(t)$ is compact in $B(W,\mH)$ for all $t\in[0,T]$ (i.e., it is relatively compact with respect to $A(0)$), then $D_{\APS}$ is Fredholm and
$$\ind(D_{\APS})=\sfl(A).$$
\end{prop}
\begin{proof}
By Lemma \ref{hatdiff}, $\hat A'$ is compact as well. This implies that $\hat A(t)- A(0)$ is compact in $B(W,\mH)$ for every $t\in[0,T]$.
From \cite[Corollary 3.5]{Les05}, it follows that
$\hat P_{<0}(t)-P_{<0}(0)$
is compact, so $(P_{<0}(0),\hat P_{<0}(t))$ is a Fredholm pair.
By Theorems \ref{inD} and \ref{main}, we get the desired result.
\end{proof}
\begin{remark}
The counterexample in the next section shows that it is not sufficient to ask for relative compactness of $A(t)-A(0)$.
\end{remark}

\subsection{A counterexample with bounded perturbation}
\label{sec:counterexample}

In this section an example is given to illustrate that $D_{\APS}$ will not always be Fredholm. There might be ``infinite exchange'' between the positive and the negative spectral subspace. This is possible, even if $A(t)$ has only discrete spectrum and its difference from $A(0)$ is bounded. 
The idea is to choose a bounded perturbation $A(t) = A(0) + B(t)$ such that the corresponding evolution operator $Q(T,0)$ interchanges the positive and negative eigenspaces of $A(0)$ and $A(T)$. 
The first step is to show that such an exchange works in a two dimensional subspace, with suitable bounds on the derivative of the perturbation. These bounds will then allow us to pass to an infinite direct sum, in which all positive and negative eigenspaces are interchanged. This means that $\Ker(D_{\APS})\cong\mH_{<0}(0)\cap Q(0,T)\mH_{\geq 0}(T)$ will be infinite-dimensional, whence $D_{\APS}$ is not Fredholm.

\begin{lemma}
\label{swap}
There exists a positive number $c>0$, such that for any
$$a=\mat{\lambda_1}{0}{0}{\lambda_2},$$
with $\lambda_1,\lambda_2\in\R$, there is a smooth family $(b(t))_{t\in[0,1]}$ of self-adjoint operators on $\C^2$ such that for $\lambda:=|\lambda_1-\lambda_2|+1$ we have
\begin{align*}
\|b(t)\|&\leq 2, & 
\|b'(t)\|&\leq c \lambda , & 
b(0)&=b(1)=0 , & 
q(1,0)e_1 &\in \spann(e_2) . 
\end{align*}
where $q$ is the evolution operator associated with $a+b(t)$, and $e_i$ denotes the $i^{\text{th}}$ standard unit vector.
\end{lemma}
\begin{proof}
Let $\phi\colon [0,1]\rightarrow [0,\frac{\pi}{2}]$ be a smooth function (chosen independently of the $\lambda_i$) satisfying  
\begin{align*}
|\phi'(t)|&\leq 2 , & 
\phi(0)&=0 , & 
\phi(1)&=\frac{\pi}{2} , & 
\phi'(0)&=\phi'(1)=0 . 
\end{align*}
Consider the self-adjoint family 
$$b(t):=\mat{0}{i\phi'(t)\exp(i(\lambda_1-\lambda_2)t)}{-i\phi'(t)\exp(i(\lambda_2-\lambda_1)t)}{0}.$$
Then the evolution operator of $a+b(t)$ is given by
$$q(t,0):=\mat{\exp(i\lambda_1t)\cos(\phi(t))}{-\exp(i\lambda_1t)\sin(\phi(t))}{\exp(i\lambda_2t)\sin(\phi(t))}{\exp(i\lambda_2t)\cos(\phi(t))}.$$
Indeed, a straightforward calculation shows that $q(0,0)=\Id$ and 
$$\ddt q(t,0)=i(a+b(t))q(t,0).$$
The required properties for $b$ are easily checked, and the requirement $q(1,0)e_1 \in \spann(e_2)$ follows since $q(1,0)$ is off-diagonal.
\end{proof}

\begin{prop}
Let $\mH:=\bigoplus\limits_{i=0}^\infty\C^2$ and let $(\lambda_i)_{i\geq0}$ be an unbounded increasing sequence of positive real numbers.
Consider the unbounded self-adjoint operator (with compact resolvents) given by 
\begin{align*}
A_0 &:= \bigoplus\limits_{i=0}^\infty a_i , &
a_i &= \mat{-\lambda_i}{0}{0}{\lambda_i} .
\end{align*}
There is a bounded family $B\colon [0,1]\rightarrow B(\mH)$ such that $A(t):=A_0+B(t)$ satisfies Assumption \ref{set} and such that $D_{\APS}$ is not Fredholm.
\end{prop}
\begin{proof}[Proof sketch]
For $i\geq0$, let $b_i$ and $q_i$ be chosen as $b$ and $q$ in Theorem $\ref{swap}$ with $\lambda_1=-\lambda_i$ and $\lambda_2=\lambda_i$. Define 
\begin{align*}
B &:= \bigoplus\limits_{i=0}^\infty b_i , &
Q &:= \bigoplus\limits_{i=0}^\infty q_i. 
\end{align*}
$Q$ is the evolution operator associated to the family $A(t):=A_0+B(t)$.
Let $\iota_i$ denote the inclusion of the $i^{\text{th}}$ summand $\C^2\into\mH$. 
For all $i\in\N$, $\iota_i(e_1)$ is a negative eigenvector of $A(0)=A_0$, but 
$$Q(1,0)\iota_i(e_1)=\iota_i(q_i(1,0)e_1)\in \spann(\iota_i(e_2))$$
is a positive eigenvector of $A(1)=A_0$ by construction. Thus 
$$\Ker(D_{\APS})\cong \Ker(\hat P_{<0}(1)_r)=  \mH_{<0}(0)\cap Q(0,1)\mH_{\geq 0}(1)=\overline{\{\spann\{\iota_i(e_1)|i\in\N\}}$$
is infinite-dimensional and hence $D_{\APS}$ is not Fredholm.
\end{proof}

% \bibliographystyle{myamsalpha}
% \bibliography{biblio}

\providecommand{\bysame}{\leavevmode\hbox to3em{\hrulefill}\thinspace}
\providecommand{\MR}{\relax\ifhmode\unskip\space\fi MR }
\providecommand{\MRhref}[2]{%
  \href{http://www.ams.org/mathscinet-getitem?mr=#1}{#2}
}
\providecommand{\href}[2]{#2}
\providecommand{\doilinktitle}[2]{#1}
\providecommand{\doilinkjournal}[2]{\href{https://doi.org/#2}{#1}}
\providecommand{\doilinkvynp}[2]{\href{https://doi.org/#2}{#1}}
\providecommand{\eprint}[2]{#1:\href{https://arxiv.org/abs/#2}{#2}}

\end{document}